\documentclass[reqno]{amsart}

\usepackage[T1]{fontenc}      % proper hyphenation of Derdzi\'nski, Ünal, ...
\usepackage[utf8]{inputenc}
\usepackage{microtype}

\usepackage{amsmath}
\usepackage{amssymb}
\everymath{\displaystyle}
\usepackage{amsthm}
\usepackage{stmaryrd}         % \varowedge (Kulkarni  Nomizu product)

\usepackage{enumitem}
\usepackage{verbatim}         % provides the "comment" environment
\usepackage{xcolor}
\usepackage{tikz}
\usetikzlibrary{arrows.meta}

\usepackage[
  colorlinks=true,
  linkcolor={black!12!blue},
  citecolor={black!35!green},
  urlcolor={black!25!blue},
  linktocpage=true,
  pdftitle={The Ricci tensor of a gradient Ricci soliton with harmonic Weyl tensor},
  pdfsubject={53C20, 53E20, 53C21, 53C25},
  pdfkeywords={Ricci solitons, Harmonic Weyl tensor, Multiply warped products}
]{hyperref}
\usepackage{orcidlink}

\theoremstyle{plain}
\newtheorem{theorem}{Theorem}[section]
\newtheorem{corollary}[theorem]{Corollary}
\newtheorem{lemma}[theorem]{Lemma}
\theoremstyle{definition}
\newtheorem{remark}[theorem]{Remark}
\newtheorem{definition}[theorem]{Definition}

\numberwithin{equation}{section}

\begin{document}

\title[Ricci tensor of gradient Ricci soliton with harmonic Weyl tensor]{The Ricci tensor of a gradient Ricci soliton with harmonic Weyl tensor}

\author[V. Borges]{V. Borges$^1$}
\author[M. A. R. M. Horácio]{M. A. R. M. Horácio$^2$}
\author[J. P. dos Santos]{J. P. dos Santos$^3$}

\address{$^1$Institute of Exact and Natural Sciences, Universidade Federal do Pará, Belém, Pará, Brazil}
\address{$^{2,3}$Mathematics Department, Universidade de Brasília, Brasília, Distrito Federal, Brazil}

\email{$^1$valterborges@ufpa.br}
\email{$^2$matheus.andrade5488@gmail.com}
\email{$^3$joaopsantos@unb.br}


\begin{abstract}
In this article, we investigate gradient Ricci solitons of dimension $n\geq4$ with harmonic Weyl tensor. We prove the equivalence of several conditions involving the Ricci eigendistributions, the geometry and the intrinsic Ricci tensor of the level sets of the potential function, the local decomposability of the corresponding net, and local multiply warped product representations. Under the local decomposability assumption, we provide a short proof of Kim's bound on the number of distinct Ricci eigenvalues which does not require the use of specialized moving frames. Combining Kim's theorem with our equivalence result, we also obtain the remaining geometric conditions without any additional hypothesis.
\end{abstract}
\maketitle

\section{Introduction and main results}\label{sec_intro}

A gradient Ricci soliton $(M,g,f,\lambda)$ is a four-tuple composed of a Riemannian manifold $(M^{n},g)$, with $n\geq3$, a smooth function $f\in C^{\infty}(M)$ and a constant $\lambda\in\mathbb{R}$ satisfying
\begin{align}\label{fundeq}
	\mathrm{Ric} +  \nabla^2 f= \lambda g.
\end{align}
These metrics have gained prominence due to their deep connection with the Ricci flow. Gradient Ricci solitons not only provide special self-similar solutions to the flow, but also frequently arise as singularity models, appearing as limits of dilations near singularities. From a dynamical perspective, they may be regarded as fixed points of the flow, modulo rescaling. Consequently, they play a central role in revealing the local and asymptotic geometry of solutions to the Ricci flow. As Catino and Mastrolia remark \cite{catinoBook}, substantial progress has been made in classifying such manifolds in dimension three, particularly with the shrinking case being completely resolved and with significant advances for the steady and expanding cases. Despite these results, a comprehensive understanding, even in three dimensions, remains elusive. This challenge becomes even more pronounced in higher dimensions, providing strong motivation for continued investigation. It is then natural to seek classifications of gradient Ricci solitons under additional curvature conditions that may enforce rigidity.

In this direction, the Einstein condition serves as a compelling first step: taking $f$ constant produces the trivial solitons $\mathrm{Ric}=\lambda g$. It is well known that Einstein manifolds have constant sectional curvature if either $n=3$, or $n\geq4$ and the Weyl tensor vanishes, the latter being equivalent to local conformal flatness. Complete gradient Ricci solitons have been classified under the assumptions that either $n=3$ and $\lambda>0$, or $n\geq4$, local conformal flatness, and $\lambda\geq0$ (see \cite{cachen,caoqian,cccmm} and references therein). Moreover, every Einstein manifold has harmonic Weyl curvature, since the Cotton tensor vanishes in the Einstein case. In this context, manifolds with harmonic Weyl curvature arise as natural generalizations of both the Einstein condition and local conformal flatness.

The classification of complete gradient Ricci solitons with harmonic Weyl tensor and $\lambda>0$ was obtained by combining the results of \cite{fega} and \cite{muse}. When $\lambda=0$ and $n=4$, these solitons were classified in \cite{kim1}. More recently, the same author achieved a classification in \cite{kim3} for $\lambda\geq0$ and $n\geq5$, recovering the results of \cite{fega} and \cite{muse} for $\lambda>0$, and also providing a local representation as a multiply warped product for any constant $\lambda\in\mathbb{R}$ and $n\geq4$. These manifolds were further studied in \cite{feng}, where a classification was obtained under additional hypotheses.

A central aspect of Kim's work \cite{kim1,kim3} is the bound on the number of distinct eigenvalues of the Ricci tensor of a gradient Ricci soliton with harmonic Weyl curvature: there can be at most three distinct eigenvalues at each point. More precisely, using a basis of eigenvectors $\{E_i\}_{i\geq1}$, with $E_1=\nabla f/|\nabla f|$, the eigenvalues $\{\lambda_i\}_{i\geq2}$ take at most two distinct values. This was proved for $n=4$ in \cite{kim1} and extended to $n\geq5$ in \cite{kim3}. This bound is crucial for several reasons:
\begin{itemize}
	\item The soliton equation and the harmonicity of the Weyl tensor, which are tensorial PDEs, reduce to a system of three ODEs for the potential function and at most two warping data.
	\item It forces the tangent bundle of each level set of $f$ to split into at most two totally umbilical distributions, ensuring that $M$ is locally a multiply warped product with at most two fibers.
	\item By analyzing the local models and using the real analyticity of $g$ and $f$ to guarantee smooth transitions between regions, one obtains the global classification by gluing together the local pieces. This approach would become substantially more complicated if four or more eigenvalues were possible.
\end{itemize}

Kim took these steps in \cite{kim1} to obtain the classification when $n=4$. As remarked in \cite{kim3}, he used an exhaustive case-by-case analysis of connection components in dimension four. It is also mentioned that a reasoning similar to the one used in \cite{kim1} becomes impractical as the dimension increases, because the number of connection coefficients grows significantly. Thus, he introduces in \cite{kim3} a refined frame field $\{F_i\}_{i=1}^n$, where $F_1=E_1$. This frame is obtained via the parallel transport of an initial ordinary orthonormal eigenframe along the $E_1$ flow and satisfies $\nabla_{F_1}F_\ell=0$ for $\ell>1$. After a few lengthy computations, this refinement provides just enough control on the Riemannian connection to push his approach through in higher dimensions. Subsequently, assuming at least three distinct $\lambda_i$ for $i\geq2$, Kim obtains contradictory identities using an $s$-invariant quantity defined by the ODEs which arise naturally in the analysis. Thus, more than three Ricci eigenvalues cannot occur.

The eigenvalue estimate we use is the following theorem of Kim:

\begin{theorem}[Kim]\label{maxxnumeig-INTRO}
Let $(M^n,g,f,\lambda)$, $n\geq4$, be a gradient Ricci soliton with harmonic Weyl curvature and nonconstant $f$. Then, the Ricci tensor of $M$ has at most three distinct eigenvalues at each point of $M$.
\end{theorem}

The main result of this article establishes the equivalence of eight local geometric and analytic conditions. These include the local decomposability of the Ricci net, a local multiply warped product structure, and the bound on the number of distinct Ricci eigenvalues proved by Kim. To formulate the first of these conditions, note that the harmonicity of the Weyl tensor implies that the eigendistributions of the Ricci tensor are integrable. Together with the line field generated by $E_1$, the eigendistributions tangent to the level sets constitute a net, which we call the \emph{Ricci net} of the soliton and denote by $\mathcal{N}$. We say that this net is locally decomposable in the sense of \cite{reckziegel} (see Definition \ref{def_loc_diff_prod}) if every point of $M$ has a neighborhood diffeomorphic to a product in such a way that the members of $\mathcal{N}$ are mapped to the distributions tangent to the corresponding factors. Equivalently, there exists a local coordinate chart simultaneously adapted to all the distributions in $\mathcal{N}$.

This hypothesis is satisfied by several well-known examples. It holds for the rigid solitons $\mathbb{R}^{k}\times N^{n-k}$ of \cite{pewy} and, more generally, for solitons with parallel Ricci tensor. It also holds for rotationally symmetric solitons, such as the Bryant soliton, and for the two fiber models $I\times_{h_1}N_1\times_{h_2}N_2$ which appear in \cite{kim1,kim3}. The Bryant soliton was generalized by Ivey to complete steady gradient Ricci solitons on doubly warped products, and the Bryant--Ivey construction was subsequently extended by Dancer and Wang to warped products over an arbitrary number of positive Einstein factors \cite{ivey1,danwan}. In each of these examples, the Ricci net is a coarsening of the canonical net of the underlying product and is therefore locally decomposable.

We now state our main theorem. Let $s$ denote the arc length parameter along the integral curves of $E_1$, let a prime denote differentiation in the $E_1$ direction, and set, for $a\geq2$,
\[
        \xi_a=\frac{\lambda-\lambda_a}{f'},\qquad f'=|\nabla f|,
\]
where $E_a$ is a unit Ricci eigenvector with eigenvalue $\lambda_a$, $D_a$ is the eigendistribution containing $E_a$, and $\Sigma_s$ is a level set of $f$. The functions $\xi_a$ are the principal curvatures of $\Sigma_s$ with respect to $E_1$.

\begin{theorem}\label{thm_equivalence}
Let $(M^n,g,f,\lambda)$, $n\geq4$, be a gradient Ricci soliton with harmonic Weyl curvature and nonconstant $f$. Then the following assertions are equivalent.
\begin{enumerate}[label=\textup{(\roman*)}]
\item\label{eq_i}   The Ricci net $\mathcal{N}$ is locally decomposable.
\item\label{eq_ii}  $M$ is locally isometric to a multiply warped product
$I\times_{h_1}N_1^{r_1}\times\cdots\times_{h_k}N_k^{r_k}$
whose one dimensional base $I$ is an integral curve of $E_1$.
\item\label{eq_iii} For every Ricci eigendistribution $D_a$ tangent to the level sets, $D_a^{\perp}$ is totally geodesic.
\item\label{eq_iv}  For every Ricci eigendistribution $D_a$ tangent to the level sets,
\[
	k_a'+2\xi_ak_a=0,
\]
where
\[
	k_a=\mathrm{Ric}_{\Sigma_s}(E_a,E_a),\qquad E_a\in D_a,\qquad |E_a|=1.
\]
\item\label{eq_v}   Each level set of $f$ is an isoparametric hypersurface with at most two distinct principal curvatures.
\item\label{eq_vi}  Each level set $\Sigma_s$ of $f$ has intrinsically parallel Ricci tensor, $\nabla^{\Sigma_s}\mathrm{Ric}_{\Sigma_s}=0$.
\item\label{eq_vii} The eigenvalues of the Ricci tensor of $M$ corresponding to eigendirections tangent to the level sets take at most two distinct values. In particular, the Ricci tensor of $M$ has at most three distinct eigenvalues.
\item\label{eq_viii} All mixed sectional curvatures of the level sets corresponding to distinct Ricci eigendistributions vanish; that is, for every pair of distinct Ricci eigendistributions $D_a$ and $D_\alpha$ tangent to the level sets, and every pair of unit vectors $E_a\in D_a$ and $E_\alpha\in D_\alpha$,
\[
\mathrm{Rm}_{\Sigma_s}(E_a,E_\alpha,E_a,E_\alpha)=0.
\]
\end{enumerate}
\end{theorem}

As stated above, Kim proved, without any additional hypothesis, that the Ricci tensor has at most three distinct eigenvalues at each point of $M$. More precisely, the eigenvalues $\{\lambda_i\}_{i\geq2}$ take at most two distinct values \cite{kim1,kim3}. As one implication of the equivalence theorem, we obtain, under the local decomposability hypothesis, a short new proof of Kim's bound that avoids specialized moving frames:

\begin{corollary}\label{cor_local_bound}
Let $(M^n,g,f,\lambda)$, $n\geq4$, be a gradient Ricci soliton with harmonic Weyl curvature and nonconstant $f$ whose Ricci net is locally decomposable. Then $M$ is locally a multiply warped product with a one dimensional base and at most two fibers, in the sense of Remark \ref{grmnt_wrpngfnctn}, and each fiber of dimension at least two is Einstein. In particular, the Ricci tensor of $M$ has at most three distinct eigenvalues at each point of $M$.
\end{corollary}

Another application of the equivalence provided by Theorem \ref{thm_equivalence} is that such a bound is, in fact, characterized by means of the geometry of level sets of the potential function and the Ricci eigendistribution. We record this consequence in the following corollary.

\begin{corollary}\label{cor_unconditional}
Let $(M^n,g,f,\lambda)$, $n\geq4$, be a gradient Ricci soliton with harmonic Weyl curvature and nonconstant $f$. Then the eigenvalue bound for the Ricci tensor of $M$ is characterized by each of the remaining conditions of Theorem \ref{thm_equivalence}.
\end{corollary}

\begin{remark}[An evolution approach to the eigenvalue bound of the Ricci tensor]\label{rem_evolution_approach}
Condition \ref{eq_iv} identifies a potential alternative route to the eigenvalue bound on the Ricci tensor. Fix a level set $\Sigma_{s_0}$ and let $\Phi_t\colon\Sigma_{s_0}\to\Sigma_{s_0+t}$ be the map induced by the local flow of $E_1$. Pulling the metrics of the nearby levels back to $\Sigma_{s_0}$, set
\[
	\widehat g_a(t)=\Phi_t^*\big(g|_{D_a}\big).
\]
Since the flow preserves $D_a$, these metrics satisfy the blockwise homothetic evolution equation
\[
	\widehat g_a'(t)=2(\xi_a\circ\Phi_t)\widehat g_a(t).
\]
An alternative evolution argument would therefore only need to establish that the intrinsic Ricci tensors are preserved in the following precise sense: after identifying the levels by $\Phi_t$, they agree as $(0,2)$-tensors on the fixed level $\Sigma_{s_0}$, namely
\[
	\Phi_t^*\mathrm{Ric}_{\Sigma_{s_0+t}}=\mathrm{Ric}_{\Sigma_{s_0}}.
\]
If this identity holds, then, because $\mathrm{Ric}_{\Sigma_s}|_{D_a}=k_ag_a$, differentiating its restriction to $D_a$ gives
\[
	0=\big((k_a'+2\xi_ak_a)\circ\Phi_t\big)\widehat g_a(t),
\]
and hence $k_a'+2\xi_ak_a=0$.
This is precisely condition \ref{eq_iv}, and the implication \ref{eq_iv}$\Rightarrow$\ref{eq_vii} would then yield the desired eigenvalue bound on the Ricci tensor.

We believe that this work points toward a broader strategy for obtaining eigenvalue bounds of the Ricci tensor in settings beyond gradient Ricci solitons. For structures with a distinguished foliation, such as gradient almost Ricci solitons and quasi-Einstein manifolds with harmonic Weyl curvature, one may seek analogous Lie evolution equations for the induced metric and its intrinsic Ricci tensor; several implications in our equivalence theorem already extend to these settings (see Remark \ref{rem_extensions}). If, as in Lemma \ref{lem_ply_deg2_}, the corresponding evolution quantity can be expressed as a nonzero polynomial of controlled degree in $\xi_a$, its vanishing forces the distinct $\xi_a$ to be roots of that polynomial. Through the corresponding structural relation between the principal curvatures and the ambient Ricci eigenvalues, this also yields an eigenvalue bound.
\end{remark}

We conclude the introduction by outlining the core of the proof of Theorem \ref{thm_equivalence}. Our approach begins by establishing the local geometric structure of the soliton before addressing the bound on the number of distinct eigenvalues of the Ricci tensor. Thus, we first obtain its local decomposition as a multiply warped product
\[
I\times_{h_1}N_1^{r_1}\times\cdots\times_{h_k}N_k^{r_k}.
\]
We then construct a nonconstant polynomial of degree at most two whose roots are $\xi_i=h_i'/h_i$, which implies that $k\leq2$. Combining this with the relation between the $\xi_i$ and the eigenvalues of the Ricci tensor, we obtain the desired result. A key observation is that neither the local geometric decomposition nor the estimate on $k$ relies on specialized moving frames. Instead, we use the integrability of the Ricci eigendistributions and the local decomposability assumption to introduce a natural system of coordinates. Namely, we use the arc length parameter $s$ along the integral curves of $E_1$, together with coordinates adapted to the integral manifolds of the Ricci eigendistributions. In these coordinates, we prove the multiply warped product structure (see Lemma \ref{lemmamany}) and construct the polynomial (see Lemma \ref{lem_ply_deg2_}). We finish the paper by proving the remaining implications of Theorem \ref{thm_equivalence}.

\section{Preliminaries}\label{prelim}

Consider vector fields $X,Y,Z,T\in\mathfrak{X}(M)$. We adopt the following conventions:
\begin{align*}
	\mathrm{Rm}(X,Y,Z,T)&=\left\langle\nabla_{Y}\nabla_{X}Z-\nabla_{X}\nabla_{Y}Z+\nabla_{[X,Y]}Z,T\right\rangle,\\
	W(X,Y,Z,T)&=\mathrm{Rm}(X,Y,Z,T)-\frac{1}{n-2}\left(\left(\mathrm{Ric}-\frac{R}{n}g\right)\varowedge g\right)(X,Y,Z,T)\\
    &\qquad-\frac{R}{2n(n-1)}(g\varowedge g)(X,Y,Z,T),\\
	C(X, Y, Z) &=\left(\nabla_X \mathrm{Ric}\right)(Y, Z) -  \left(\nabla_Y \mathrm{Ric}\right)(X, Z)\\
    &\qquad-\frac{1}{2(n-1)} \left\{\left(\nabla_X (Rg)\right)(Y, Z) - \left(\nabla_Y (Rg)\right)(X, Z)\right\},
\end{align*}
where $\varowedge$ is the Kulkarni  Nomizu product \cite[page 47]{besse}. With these conventions $\mathrm{Rm}(X,Y,X,Y)$ is the sectional curvature of $\mathrm{span}\{X,Y\}$ for orthonormal $X,Y$. Recall that for $n\geq4$ the Weyl tensor is harmonic if and only if the Cotton tensor $C$ vanishes, equivalently if and only if the Schouten tensor $\mathcal{A}=\mathrm{Ric}-\frac{R}{2(n-1)}g$ is a Codazzi tensor.

\subsection{Results for when the Cotton tensor vanishes}\label{subsec_cotton}

We start by recalling some results from \cite{fega} and \cite{kim1} which will be useful in the rest of the paper. In the version of the next lemma, we emphasize the fact that it is an equivalence.

\begin{lemma}[\cite{fega}]\label{bari}
	Let $(M^n, g, f, \lambda)$ be a gradient Ricci soliton. Then the Cotton tensor of $M$ vanishes if and only if
	\begin{align}\label{eq_bari}
		\mathrm{Rm}(\nabla f,X,Y,Z)=Y\left(\frac{R}{2(n-1)}\right)g(X,Z)-Z\left(\frac{R}{2(n-1)}\right)g(X,Y)
	\end{align}
	for all $X,Y,Z \in \mathfrak{X}(M)$. If $n\geq4$, this is equivalent to $M$ having harmonic Weyl curvature.
\end{lemma}

\begin{proof}
	It is easy to see that on a gradient Ricci soliton we have
	\[
		\left(\nabla_X \mathrm{Ric}\right)(Y,Z)
		=g\left(\nabla_{\nabla_XY}\nabla f-\nabla_X\nabla_Y\nabla f,Z\right).
	\]
	Consequently,
	\[
		\left(\nabla_X\mathrm{Ric}\right)(Y,Z)-\left(\nabla_Y\mathrm{Ric}\right)(X,Z)
		=\mathrm{Rm}(X,Y,\nabla f,Z).
	\]
	Straightforward computations also yield
	\[
		\left(\nabla_X(Rg)\right)(Y,Z)-\left(\nabla_Y(Rg)\right)(X,Z)
		=X(R)g(Y,Z)-Y(R)g(X,Z).
	\]
	Therefore, using the definition of the Cotton tensor, $C$ vanishes if and only if
	\[
			\mathrm{Rm}( \nabla f, Z,Y,X ) = Y\left( \tfrac{R}{2(n-1)} \right) g(X, Z) - X \left( \tfrac{R}{2(n-1)}\right) g(Y,Z),
	\]
	which is equivalent to \eqref{eq_bari}.
\end{proof}

\begin{lemma}[Lemma 2.2 of \cite{fega}]\label{caochen}
	Let $(M,g,f,\lambda)$ be a gradient Ricci soliton with zero Cotton tensor and nonconstant $f$, let $c$ be a regular value of $f$ and $\Sigma_{c}=f^{-1}(c)$. Then:
		\begin{enumerate}
			\item\label{cao1} where $\nabla f\neq0$, $E_{1}=\frac{\nabla f}{|\nabla f|}$ is an eigenvector of $\mathrm{Ric}$;
			\item\label{cao2} $|\nabla f|$ is constant on each connected component of $\Sigma_{c}$;
			\item\label{cao3} there is a locally defined function $s$ with $\mathrm{d} s=\frac{\mathrm{d} f}{|\nabla f|}$ and $E_{1}=\nabla s$;
			\item\label{cao4} $E_{1}E_{1}f=-\mathrm{Ric}(E_{1},E_{1})+\lambda$; in particular $\lambda_1 = \mathrm{Ric}(E_{1},E_{1})$ is constant on each connected component of $\Sigma_{c}$;
			\item $\nabla_{E_1}E_1=0$.
		\end{enumerate}
\end{lemma}

Let $\mathcal{A}$ be a Codazzi tensor, let $E_{\mathcal{A}}(x)$ be the number of its distinct eigenvalues at $x$, and let
\begin{align}\label{eq_MA}
	M_{\mathcal{A}}=\{x\in M \ : \ E_{\mathcal{A}}(x)\text{ is constant in a neighborhood of }x\},
\end{align}
an open dense subset of $M$ on which the eigenvalues of $\mathcal{A}$ are smooth functions $\lambda_{1},\ldots,\lambda_{n}$ \cite{derd}. Since the Schouten tensor and the Ricci tensor differ by a multiple of $g$, they have the same eigendistributions, and we index the eigenvalues of $\mathrm{Ric}$ accordingly. The next lemma provides the identities we shall need (see \cite{kim3} for the formulation below).

\begin{lemma}[Derdzi\'{n}ski \cite{derd}]\label{derdlemma}
		Let $(M^n,g)$, $n\geq4$, have harmonic Weyl curvature and let $\{E_{i}\}_{i=1}^{n}$ be a local orthonormal frame with $\mathrm{Ric}(E_{i},\cdot)=\lambda_{i}g(E_{i},\cdot)$ on $M_{\mathcal{A}}$. Then:
		\begin{enumerate}
			\item\label{derd2} if $i$, $j$ and $k$ are pairwise distinct, then
			$(\lambda_{j}-\lambda_{k})\left\langle\nabla_{E_{i}}E_{j},E_{k}\right\rangle=(\lambda_{i}-\lambda_{k})\left\langle\nabla_{E_{j}}E_{i},E_{k}\right\rangle$;
			\item\label{derd3} given distinct eigenfunctions $\mu$ and $\nu$ of $\mathcal{A}$ and local vector fields $U$ and $V$ such that $\mathcal{A}(V)=\mu V$ and $\mathcal{A}(U)=\nu U$, with $|U|=1$, one has
			\[
				V(\nu)=(\nu-\mu)\left\langle\nabla_UU,V\right\rangle;
			\]
			\item\label{derd4} each distribution $D_{\lambda_{i}}(p)=\{v\in T_{p} M \ : \ \mathrm{Ric}(v,\cdot)=\lambda_{i}g(v,\cdot)\}$ is integrable with totally umbilical leaves.
		\end{enumerate}
\end{lemma}

The last result we recall is the constancy of the remaining eigenvalues on the level sets. It is proved for $n=4$ in \cite{kim1}, and Kim observes in \cite{kim3} that the argument extends to all dimensions; a proof along the same lines, in the quasi-Einstein setting, is given in \cite[Lemma 3]{shin}. This result is independent of the eigenvalue bound in \cite{kim1,kim3}.

\begin{lemma}[\cite{kim1,kim3}]\label{lem_lambda_level}
		Let $(M^n,g,f,\lambda)$, $n\geq4$, be a gradient Ricci soliton with harmonic Weyl curvature and nonconstant $f$. Then $\lambda_{2},\ldots,\lambda_{n}$ are constant on each connected component of a regular level set $\Sigma_{c}$ of $f$. Consequently, the functions $\xi_{2},\ldots,\xi_{n}$ defined by \eqref{eq_q} are also constant on each connected component of $\Sigma_c$.
\end{lemma}

\subsection{Multiply warped products}\label{subsec_mwp}

A \emph{multiply warped product} is the product manifold $B\times N_1^{r_1} \times \cdots \times N_{k}^{r_k}$ equipped with
\begin{align}\label{MWPS}
	M = B \times_{h_1} N_1^{r_1} \times \cdots \times_{h_k} N_{k}^{r_k},\qquad
	g=g_{B}+h_1^2 g_{N_1}+\cdots+h_{k}^2 g_{N_k},
\end{align}
where each warping function $h_{i}:B\rightarrow(0,\infty)$ is smooth. We call $B$ the base and $N_i^{r_i}$ the fibers; for $k=1$ this is the classical warped product \cite{oneill}. Such metrics occur throughout the literature on curvature conditions \cite{bishop,brovaz,choui,danwan} and in the classification of solitons \cite{caoqian,cccmm,caozhou,kim1,kim3,feng,shin}.

\begin{remark}[Counting fibers]\label{grmnt_wrpngfnctn}
	As in \cite{brovaz}, we always normalize a multiply warped product so that its warping functions are pairwise nonproportional: if $h_i=c\,h_j$ with $c>0$, we rescale $g_{N_j}$ and merge the two factors into a single fiber. The integer $k$ in \eqref{MWPS} is then called the \emph{number of fibers}. A constant warping function corresponds to a Riemannian product factor and, when convenient, may be absorbed into the base.
\end{remark}

The next lemma records the connection and Ricci formulas \cite[Propositions 2.5 and 2.6]{DobUn}.

\begin{lemma}[\cite{DobUn}]\label{warped}
	Let $M=B\times_{h_1}N_1^{r_1}\times\cdots\times_{h_k}N_k^{r_k}$ be a multiply warped product, and let $X,Y\in\mathcal{L}(B)$, $V\in\mathcal{L}(N_i)$ and $W\in\mathcal{L}(N_j)$ be lifted vector fields. Then:
	\begin{enumerate}
		\item the Levi-Civita connection satisfies
		\begin{equation}\label{eq_warped_conn}
		\begin{aligned}
			\nabla_XY&=\nabla_X^BY,\\
			\nabla_XV=\nabla_VX&=\frac{X(h_i)}{h_i}V,\\
			\nabla_WV&=0 &&\text{if }i\neq j,\\
			\nabla_WV&=\nabla_W^{N_i}V-h_i g_{N_i}(W,V)\nabla^Bh_i &&\text{if }i=j;
		\end{aligned}
		\end{equation}
		\item the Ricci tensor is given by
		\begin{equation}\label{eq_warped_ric}
		\begin{aligned}
			\mathrm{Ric}(X,Y)&=\mathrm{Ric}_B(X,Y)-\sum_{1\leq\ell\leq k}\frac{r_\ell}{h_\ell}\nabla_B^2h_\ell(X,Y),\\
			\mathrm{Ric}(X,V)&=0,\\
			\mathrm{Ric}(V,W)&=0 &&\text{if }i\neq j,\\
			\mathrm{Ric}(V,W)&=\mathrm{Ric}_{N_i}(V,W)
			-\left(\frac{\Delta_Bh_i}{h_i}+(r_i-1)\frac{|\nabla_Bh_i|^2}{h_i^2}\right.\\
			&\qquad\left.+\sum_{\substack{1\leq\ell\leq k\\ \ell\neq i}}r_\ell
			\frac{g_B(\nabla_Bh_i,\nabla_Bh_\ell)}{h_ih_\ell}\right)g(V,W),
			\qquad\text{if }i=j;
		\end{aligned}
		\end{equation}
	\end{enumerate}
\end{lemma}

\subsection{Nets and local decomposability}\label{subsec_nets}

\begin{definition}[\cite{reckziegel}]\label{def_loc_diff_prod}
A \emph{net} on a manifold $P$ is a family $\{D_i\}_{i=0}^{k}$ of integrable distributions with $TP=\bigoplus_{i=0}^{k} D_i$. The \emph{canonical net} of a product $P_0\times\cdots\times P_k$ is the family of distributions tangent to its factors. The net $\{D_i\}$ is \emph{locally decomposable} if every $p\in P$ has a neighbourhood $U$ admitting a diffeomorphism $\psi\colon U\to P_0\times\cdots\times P_k$ onto a product manifold such that, for every $q\in U$ and every $i$, $\mathrm{d}\psi\big(D_i(q)\big)$ is the tangent space at $\psi(q)$ of the $i$ th factor.
\end{definition}

Equivalently, $\{D_i\}$ is locally decomposable exactly when every point admits a chart $(x^0,\ldots,x^k)$, with $x^i$ a block of $\dim D_i$ coordinates, such that $D_i$ is spanned by the coordinate fields of the $i$ th block. We shall use the following local form of the criterion in \cite{reckziegel}. For a family of distributions $\{D_i\}_{i=0}^{k}$ with $TP=\bigoplus_{i=0}^{k}D_i$, set $\widehat{D}_i=\bigoplus_{j\neq i}D_j$. Then the following conditions are equivalent:
\begin{enumerate}[label=\textup{(\alph*)}]
\item $\{D_i\}$ is a locally decomposable net;
\item $\widehat{D}_i$ is integrable for every $i$;
\item $[D_i,D_j]\subset D_i\oplus D_j$ for all $i,j$.
\end{enumerate}
Indeed, the first condition implies the third because the coordinate fields of distinct blocks commute in an adapted chart. The third implies the second because $D_i\oplus D_j\subset\widehat{D}_m$ whenever $m\neq i,j$, so each $\widehat{D}_m$ is closed under brackets. Finally, if the second condition holds, each $\widehat{D}_i$ is locally the kernel of a submersion $\pi_i\colon U\to\mathbb{R}^{\dim D_i}$, and $\psi=(\pi_0,\ldots,\pi_k)$ is a local product chart adapted to the $D_i$. Thus the bracket of vector fields tangent to two members of a locally decomposable net has no component in a third member.

\section{Local representation as a multiple warped product}\label{cotton_results}

From now on, $(M^n,g,f,\lambda)$, $n\geq4$, is a gradient Ricci soliton with harmonic Weyl curvature and nonconstant $f$. Let
\[
	\mathcal{R}=\{x\in M \ : \ \nabla f(x)\neq0\}.
\]
Since $g$ and $f$ are real analytic in harmonic coordinates \cite{ivey2}, $\mathcal{R}$ is dense in $M$.

All arguments below are local. On $\mathcal{R}\cap M_{\mathcal{A}}$, with $M_{\mathcal{A}}$ as in \eqref{eq_MA}, we fix an orthonormal eigenframe $\{E_{i}\}_{i=1}^n$ as in Lemma \ref{derdlemma}, with $E_{1}=\frac{\nabla f}{|\nabla f|}$ by Lemma \ref{caochen}\eqref{cao1}. We write $s$ for the arc length parameter of Lemma \ref{caochen}\eqref{cao3}, $f'=|\nabla f|$, and a prime for differentiation in the $E_1$ direction.

Following \cite{feng}, indices $a,b,c,\ldots$ and $\alpha,\beta,\gamma,\ldots$ always range in $\{2,\ldots,n\}$, and we set
\begin{align}\label{eq_class}
	[a]=\{j\in\{2,\ldots,n\} \ : \ \lambda_{j}=\lambda_{a}\},\qquad r_a=\#[a],\qquad D_{a}=\mathrm{Span}\{E_{j} \ : \ j\in[a]\};
\end{align}
the convention is that $b,c\in[a]$, that $\beta,\gamma\in[\alpha]$, and that $[a]\neq[\alpha]$. We denote by $\ell$ the number of distinct classes $[a]$, so that the \emph{Ricci net}
\[
	\mathcal{N}=\big\{\langle E_1\rangle,\ D_{a_1},\ldots,D_{a_\ell}\big\}
\]
has $\ell+1$ members, where $a_1,\ldots,a_\ell$ are representatives of the classes. Each $D_a$ is integrable by Lemma \ref{derdlemma}\eqref{derd4} and the integrability of the level sets, and $\langle E_1\rangle$ is integrable because it has rank one, so $\mathcal{N}$ is indeed a net.

\subsection{\texorpdfstring{$R$ and the Eigenvalues of $\mathrm{Ric}$ depend only on $s$}{R and the eigenvalues of Ric depend only on s}}\label{subsec_dep_s}

We first record the identities available with no hypothesis beyond \eqref{fundeq} and the harmonicity of the Weyl tensor. Set
\begin{align}\label{eq_q}
	\xi_{a}=\frac{\lambda-\lambda_{a}}{f'},\qquad H=\sum_{a\geq2}\xi_{a},\qquad q=-\frac{R'}{2(n-1)f'} .
\end{align}
By Lemmas \ref{caochen} and \ref{lem_lambda_level}, the functions $R$, $f'$, $\lambda_1$, $\lambda_a$, $\xi_a$, $H$ and $q$ are constant on the connected components of the level sets, hence depend only on $s$.

\begin{lemma}\label{nablaframe}
	For every $X\in D_a$, $a\geq2$, we have
	\begin{align}\label{step2}
		\nabla_XE_1=\xi_aX.
	\end{align}
	In particular, $\xi_a=-\left\langle\nabla_{E_a}E_a,E_1\right\rangle$.
\end{lemma}

\begin{proof}
For $X\in D_a$, equation \eqref{fundeq} gives
\[
	\nabla_{X}E_{1}=\frac{\nabla_{X}\nabla f}{|\nabla f|}=\frac{\lambda X-\mathrm{Ric}(X,\cdot)^{\sharp}}{|\nabla f|}=\frac{(\lambda-\lambda_{a})X}{f'}=\xi_{a}X,
\]
and the last assertion follows from $\left\langle\nabla_{E_a}E_1,E_a\right\rangle=-\left\langle\nabla_{E_a}E_a,E_1\right\rangle$.
\end{proof}

We shall also use that $\nabla_{E_1}$ preserves each $D_a$. Indeed, if $[a]\neq[\alpha]$, Lemma \ref{derdlemma}\eqref{derd2} and \eqref{step2} give
\[
	(\lambda_a-\lambda_\alpha)\left\langle\nabla_{E_1}E_a,E_\alpha\right\rangle=0.
\]
Together with $\nabla_{E_1}E_1=0$, this shows that
\begin{align}\label{eq_preserve_Da}
	\nabla_{E_1}D_a\subset D_a,\qquad [E_1,D_a]\subset D_a.
\end{align}
Since $R$ depends only on $s$, Lemma \ref{bari} also gives
\begin{align}\label{eq_radial_curvature}
	\mathrm{Rm}(E_1,X,E_1,Y)=-q\langle X,Y\rangle
\end{align}
for $X,Y$ tangent to a level set. Moreover, \cite{feng} gives
\begin{align}\label{excep}
	\xi_a'+\xi_a^2=q.
\end{align}

\subsection{The intrinsic Ricci curvature of the level sets}\label{subsec_level}

Let $\Sigma_s$ denote a connected component of a level set of $f$, with the induced metric. By Lemma \ref{nablaframe}, the shape operator of $\Sigma_s$ with respect to $E_1$ is $\xi_a\,\mathrm{id}$ on $D_a$. Thus, the $\xi_a$ are the principal curvatures of $\Sigma_s$ and $H$ is its mean curvature. It follows from Lemma \ref{lem_lambda_level} that they are constant on $\Sigma_s$, so every level set is an isoparametric hypersurface. It is convenient to introduce, as in \cite{feng},
\begin{align}\label{def_B_C}
        B=\frac{(n-1)\lambda-R+\lambda_{1}-(f')^2}{f'},\qquad C=q+\lambda .
\end{align}
Since $\sum_{a\geq2}\lambda_a=R-\lambda_1$, definition \eqref{eq_q} gives at once
\begin{align}\label{eq_BH}
	B=H-f' .
\end{align}

The frame $\{E_a\}_{a\geq2}$ also diagonalizes $\mathrm{Ric}_{\Sigma_s}$, and
\begin{align}\label{eq_full_ricci_decomposition}
	k_a:=\mathrm{Ric}_{\Sigma_s}(E_a,E_a)=\lambda_a+\xi_a'+H\xi_a=C+B\xi_a-\xi_a^{2}.
\end{align}
Indeed, writing $\mathcal{S}=\nabla E_1$ for the shape operator of $\Sigma_s$ with respect to $E_1$, the trace of the Gauss equation reads
\[
	\mathrm{Ric}_{\Sigma_s}(X,Y)=\mathrm{Ric}(X,Y)-\mathrm{Rm}(E_1,X,E_1,Y)+H\langle\mathcal{S}X,Y\rangle-\langle\mathcal{S}X,\mathcal{S}Y\rangle
\]
for $X,Y$ tangent to $\Sigma_s$, the sign of the second term being fixed by our convention, in which $\mathrm{Rm}(E_1,X,E_1,X)$ is the sectional curvature of $\mathrm{span}\{E_1,X\}$ for unit $X\perp E_1$. If $X=E_a$ and $Y=E_j$ with $j\neq a$, every term vanishes: the first because the frame diagonalizes $\mathrm{Ric}$, the second by \eqref{eq_radial_curvature}, and the last two because $\mathcal{S}$ is diagonal in this frame. Taking instead $X=Y=E_a$ and using \eqref{eq_radial_curvature} again,
\[
	k_a=\lambda_a+q+H\xi_a-\xi_a^{2}=\lambda_a+\xi_a'+H\xi_a .
\]
Finally $\lambda_a=\lambda-f'\xi_a$, so with \eqref{def_B_C} and \eqref{eq_BH} we get $k_a=(\lambda+q)+(H-f')\xi_a-\xi_a^{2}=C+B\xi_a-\xi_a^{2}$. All the quantities involved depend only on $s$.
Consequently $\mathrm{Ric}_{\Sigma_s}=\sum_{[a]}k_a\,g_a$, where $g_a$ is the restriction of $g$ to $D_a$, and each $k_a$ is constant on $\Sigma_s$.

\subsection{Local multiply warped product structure}\label{der_ivey_ex}

We now prove the equivalence of conditions \ref{eq_i}, \ref{eq_ii} and \ref{eq_iii} of Theorem \ref{thm_equivalence}.

\begin{lemma}\label{lemmamany}
		Suppose that $\mathcal{N}$ is locally decomposable. Let $(x_b)_{b\in[a]}$ and $(x_\beta)_{\beta\in[\alpha]}$ be local coordinate systems on the integral manifolds of $D_a$ and $D_\alpha$, respectively. Then every point has a coordinate neighborhood $V$ in which these coordinates form part of a common adapted coordinate system $(x^1,\ldots,x^n)$, $\partial_1=E_1$, the fields $\partial_a$ with $a\in[a]$ span $D_a$, and
	\begin{align}\label{37lemma36}
		\partial_{1}g_{ab}=2\xi_{a}g_{ab},\qquad \partial_\alpha g_{ab}=0 .
	\end{align}
\end{lemma}

\begin{proof}
	By Definition \ref{def_loc_diff_prod} there is a chart $(y^1,x^2,\ldots,x^n)$ adapted to $\mathcal{N}$, with $\partial_{y^1}$ spanning $\langle E_1\rangle$ and $\{\partial_{a}\}_{a\in[a]}$ spanning $D_a$ for each class. The function $s$ of Lemma \ref{caochen}\eqref{cao3} satisfies $\mathrm{d}s=\langle E_1,\cdot\rangle$, hence $\partial_a(s)=0$ for $a\geq2$; thus $s$ depends only on $y^1$ and, since $\mathrm{d}s(\partial_{y^1})\neq0$, we may replace $y^1$ by $x^1=s$ without disturbing the adapted structure. In the new chart $\partial_1$ is a unit field spanning $\langle E_1\rangle$, so $\partial_1=E_1$. As the members of $\mathcal{N}$ are pairwise orthogonal, $g$ is block diagonal in this chart, and $g_{ab}=0$ unless $[a]=[b]$.

	The first identity in \eqref{37lemma36} follows from Lemma \ref{nablaframe}, since $[\partial_1,\partial_a]=0$ and $\xi_a=\xi_b$ for $b\in[a]$:
	\[
		\partial_{1}g_{ab}=\langle\nabla_{\partial_{1}}\partial_{a},\partial_{b}\rangle+\langle\partial_{a},\nabla_{\partial_{1}}\partial_{b}\rangle=\langle\nabla_{\partial_{a}}\partial_{1},\partial_{b}\rangle+\langle\partial_{a},\nabla_{\partial_{b}}\partial_{1}\rangle=2\xi_{a}g_{ab}.
	\]
	For the second, fix $[a]\neq[\alpha]$; we claim that $\nabla_{\partial_{\alpha}}\partial_{a}=0$. Since $[\partial_a,\partial_\alpha]=0$, Lemma \ref{nablaframe} gives, for any $z\in\{1,\ldots,n\}$,
	\begin{align*}
		\mathrm{Rm}(\partial_{a},\partial_{\alpha},\partial_1,\partial_{z})
		&=\langle\nabla_{\partial_{\alpha}}\nabla_{\partial_{a}}\partial_1-\nabla_{\partial_{a}}\nabla_{\partial_{\alpha}}\partial_1,\partial_{z}\rangle\\
		&=\langle\nabla_{\partial_{\alpha}}(\xi_a\partial_{a})-\nabla_{\partial_{a}}(\xi_\alpha\partial_{\alpha}),\partial_{z}\rangle
		=(\xi_{a}-\xi_{\alpha})\langle\nabla_{\partial_{\alpha}}\partial_{a},\partial_{z}\rangle,
	\end{align*}
	where the last step uses $\nabla_{\partial_{\alpha}}\partial_{a}=\nabla_{\partial_{a}}\partial_{\alpha}$ together with $\partial_\alpha(\xi_a)=\partial_a(\xi_\alpha)=0$, valid because $\partial_a$ and $\partial_\alpha$ are tangent to the level sets, along which $\xi_a$ and $\xi_\alpha$ are constant.
	On the other hand, by the symmetry $\mathrm{Rm}(X,Y,Z,T)=\mathrm{Rm}(Z,T,X,Y)$ and \eqref{eq_bari},
	\begin{align*}
		\mathrm{Rm}(\partial_{a},\partial_{\alpha},\partial_1,\partial_{z})&=\frac{1}{f'}\mathrm{Rm}(\nabla f,\partial_{z},\partial_{a},\partial_{\alpha})\\
		&=\frac{1}{f'}\left(\partial_{a}\Big(\tfrac{R}{2(n-1)}\Big)g_{z\alpha}-\partial_{\alpha}\Big(\tfrac{R}{2(n-1)}\Big)g_{za}\right)=0,
	\end{align*}
	because $R$ is constant on the level sets. As $\xi_{a}\neq\xi_{\alpha}$, we conclude that $\langle\nabla_{\partial_{\alpha}}\partial_{a},\partial_{z}\rangle=0$ for every $z$, that is, $\nabla_{\partial_{\alpha}}\partial_{a}=0$. Therefore $\partial_{\alpha}g_{ab}=\langle\nabla_{\partial_{\alpha}}\partial_{a},\partial_{b}\rangle+\langle\partial_{a},\nabla_{\partial_{\alpha}}\partial_{b}\rangle=0$.
\end{proof}

We now use Lemma \ref{lemmamany} to obtain the multiply warped product structure.

\begin{theorem}\label{decompwarp}
Let $(M^n,g,f,\lambda)$, $n\geq4$, be a gradient Ricci soliton with harmonic Weyl curvature and nonconstant $f$. The following are equivalent:
\begin{enumerate}[label=\textup{(\roman*)}]
\item the Ricci net $\mathcal{N}$ is locally decomposable;
\item $M$ is locally isometric to a multiply warped product
\begin{align}\label{eq_local_model}
	I\times_{h_1}N_1^{r_1}\times\cdots\times_{h_k}N_k^{r_k}
\end{align}
whose one dimensional base $I$ is an integral curve of $E_1$;
\setcounter{enumi}{2}
\item $D_a^{\perp}$ is totally geodesic for every $a\geq2$.
\end{enumerate}
Furthermore, every fiber with $r_i\geq2$ is Einstein. In the decomposition obtained from the Ricci net, the base parameter is the arc length $s$, the fibers are the leaves of the distributions $D_{a_i}$, and
\begin{align}\label{eqXiRelFuncWarp}
	\frac{\lambda-\lambda_{a_i}}{f'}=\xi_{i}=\frac{h_{i}'}{h_i}.
\end{align}
\end{theorem}

\begin{proof}
\emph{(i)$\Rightarrow$(ii).} Fix a coordinate neighborhood $V$ as in Lemma \ref{lemmamany}, so that $V\cong I\times N_1\times\cdots\times N_\ell$ as a manifold, with $N_i$ an integral manifold of $D_{a_i}$ and $I$ an integral curve of $E_1$ parametrized by $s$. Fix $i$ and $s_0\in I$ and set
\[
	h_i(s)=\exp\left(\int_{s_0}^s\xi_i(y)\,\mathrm{d}y\right),
\]
where $\xi_i=\xi_{a_i}$; this is a positive function of $s$ alone satisfying \eqref{eqXiRelFuncWarp}. By \eqref{37lemma36}, for $a,b\in[a_i]$,
\[
	\partial_1\big(h_i^{-2}g_{ab}\big)=0\qquad\text{and}\qquad \partial_\alpha\big(h_i^{-2}g_{ab}\big)=0 \quad\text{for } [\alpha]\neq[a_i],
\]
so $h_i^{-2}g_{ab}$ depends only on the coordinates of $N_i$ and equals its value at $s_0$. Writing $(g_{N_i})_{ab}=g_{ab}(s_0,\cdot)$, which is a Riemannian metric on $N_i$, we obtain $g_{ab}=h_i^2(g_{N_i})_{ab}$. Since $g$ is block diagonal with $g_{11}=1$, this gives
\[
	g=\mathrm{d}s^2+h_1^2g_{N_1}+\cdots+h_\ell^2g_{N_\ell}
\]
on $V$, which is \eqref{eq_local_model}. The warping functions are pairwise nonproportional because the $\xi_i$ are pairwise distinct.

We next show that each fiber with $r_i\geq2$ is Einstein. Since $\nabla f$ is tangent to $I$, the function $f$ depends only on $s$. Let $V,W$ be lifts of vector fields on $N_i$. By \eqref{eq_warped_conn}, $\nabla^2f(V,W)=-\langle\nabla_VW,\nabla f\rangle=f'h_ih_i'\,g_{N_i}(V,W)$, so \eqref{fundeq} and \eqref{eq_warped_ric} give
\begin{align}\label{eq_fiber_einstein}
	\mathrm{Ric}_{N_{i}}(V,W)= \Big(h_ih''_i + (r_i - 1)(h'_i)^2+ h_ih'_i \sum_{j\neq i} r_{j}\frac{h'_{j}}{h_{j}}-f'h_ih'_i+\lambda h_i^2 \Big) g_{N_{i}}(V, W).
\end{align}
The coefficient of $g_{N_i}$ depends only on $s$, while $\mathrm{Ric}_{N_i}$ and $g_{N_i}$ do not depend on $s$ at all. Hence both sides are constant multiples of $g_{N_i}$, and we may write
\begin{align}\label{eq_mu_i}
	\mathrm{Ric}_{N_i}=(r_i-1)\mu_i\,g_{N_i},\qquad \mu_i\in\mathbb{R},
\end{align}
a normalization which also covers the case $r_i=1$, where both sides vanish identically.

\emph{(ii)$\Rightarrow$(iii).} In a multiply warped product, $\nabla_XY=\nabla^B_XY$ for $X,Y$ tangent to the base, $\nabla_XV=\frac{X(h_j)}{h_j}V$ for $V$ tangent to $N_j$, $\nabla_WV=0$ for $V,W$ tangent to distinct fibers, and $\nabla_WV\in TN_j\oplus TB$ for $V,W$ tangent to the same fiber $N_j$ \cite[Proposition 2.5]{DobUn}. Hence the distribution $TB\oplus\bigoplus_{j\in J}TN_j$ is totally geodesic for any set of indices $J$. Since the base is an integral curve of $E_1$, the field $\nabla f$ is tangent to it, so $f$ depends only on $s$; the argument of the previous paragraph therefore applies and shows that the fibers of dimension at least two are Einstein. By \eqref{eq_warped_ric} the Ricci tensor of $M$ is then diagonal with respect to the splitting $\langle\partial_s\rangle\oplus TN_1\oplus\cdots\oplus TN_k$, so each $D_a$ is a sum of some of the $TN_j$ and $D_a^\perp=\langle\partial_s\rangle\oplus\bigoplus_{TN_j\not\subset D_a}TN_j$ is totally geodesic.

\emph{(iii)$\Rightarrow$(i).} A totally geodesic distribution is integrable, so $D_a^\perp$ is integrable for every $a\geq2$. The only remaining complementary distribution in the local decomposability criterion is $\bigoplus_{a}D_a=\langle E_1\rangle^\perp$, which is tangent to the level sets of $f$ and hence integrable. Thus $\mathcal{N}$ is locally decomposable.

The identification of the base and the fibers and \eqref{eqXiRelFuncWarp} follow from the construction in (i)$\Rightarrow$(ii), and the Einstein property of the fibers follows from \eqref{eq_fiber_einstein}.
\end{proof}

\section{Estimating the number of fibers}\label{number_fiber}

In this section, we prove that the number of fibers in the representation given by Theorem \ref{decompwarp} is at most two. When the Ricci net is locally decomposable, we use the local model \eqref{eq_local_model} together with \eqref{eqXiRelFuncWarp} and \eqref{eq_mu_i}, and write $\xi_i=\xi_{a_i}$, $k_i=k_{a_i}$ and $r_i=r_{a_i}$.

A level set of $f$ in this model is the Riemannian product $N_1\times\cdots\times N_\ell$ with metric $h_1^2g_{N_1}+\cdots+h_\ell^2g_{N_\ell}$, the $h_i$ being constants on it. Hence its Ricci tensor restricted to $TN_i$ is $\mathrm{Ric}_{N_i}=(r_i-1)\mu_ig_{N_i}$ by \eqref{eq_mu_i}. A unit vector $E_a$ with $a\in[a_i]$ has $g_{N_i}(E_a,E_a)=h_i^{-2}$, so \eqref{eq_full_ricci_decomposition} gives
\begin{align}\label{first_scnd_Dg}
	k_i=C+B\xi_i-\xi_i^{2}=(r_i-1)\frac{\mu_i}{h_i^{2}},
	\qquad
	\xi_i^2-B\xi_i-C=-(r_i-1)\frac{\mu_i}{h_i^2}.
\end{align}
Since $\mu_i$ and $r_i$ are constants and $h_i'/h_i=\xi_i$ by \eqref{eqXiRelFuncWarp}, differentiation gives
\begin{align}\label{eq_ode_k}
	k_i'+2\xi_ik_i=0 .
\end{align}

\begin{lemma}\label{lem_ply_deg2_}
For every $a\geq2$,
\begin{align}\label{eq_key_identity}
	k_a'+2\xi_ak_a=P(\xi_a),
\end{align}
where
	\begin{align}\label{poly_atmst2}
        P(x)=Bx^2+(B'+2\lambda)x+(C-\lambda)B+C'
	\end{align}
is not the zero polynomial. Consequently, in the local representation \eqref{eq_local_model}, the functions $\xi_1,\ldots,\xi_\ell$ satisfy
\[
	P(\xi_i)=B\xi_i^2+(B'+2\lambda)\xi_i+(C-\lambda)B+C'=0.
\]
\end{lemma}

\begin{proof}
By \eqref{eq_full_ricci_decomposition}, $k_a=C+B\xi_a-\xi_a^{2}$, and by \eqref{excep} and \eqref{def_B_C}, $\xi_a'=-\xi_a^{2}+C-\lambda$. Therefore
\[
	k_a'=C'+B'\xi_a+(B-2\xi_a)\big(-\xi_a^{2}+C-\lambda\big),\qquad
	2\xi_ak_a=2C\xi_a+2B\xi_a^{2}-2\xi_a^{3}.
\]
Adding these expressions and simplifying gives \eqref{eq_key_identity}.

	Suppose now that $P$ is the zero polynomial. Then $B\equiv0$, so $B'\equiv0$ and the coefficient of $x$ gives $\lambda=0$; the constant term then gives $C'\equiv0$, so $C\equiv c_0$ for some $c_0\in\mathbb{R}$. Since $\lambda=0$, definition \eqref{def_B_C} yields $q\equiv c_{0}$.

	On a gradient Ricci soliton one has $\nabla R=2\,\mathrm{Ric}(\nabla f,\cdot)^{\sharp}$ (see, e.g., \cite{caozhou,catinoBook}), so $R'=2\lambda_1f'$ and therefore $q=-\lambda_1/(n-1)$. Combining this with $q\equiv c_0$ and with Lemma \ref{caochen}\eqref{cao4},
	\[
		\lambda_1=-(n-1)c_0,\qquad f''=\lambda-\lambda_1=(n-1)c_0 .
	\]
	Next, $B\equiv0$ and \eqref{eq_BH} give $H=f'$, so that $H'=f''=(n-1)c_0$. On the other hand, summing \eqref{excep} over the $n-1$ indices $a\geq2$ gives
	\[
		H'=\sum_{a\geq2}\xi_a'=\sum_{a\geq2}\big(q-\xi_a^{2}\big)=(n-1)c_{0}-\sum_{a\geq2}\xi_a^{2}.
	\]
	Comparing the two expressions for $H'$ yields $\sum_{a\geq2}\xi_a^{2}=0$, so $\xi_a=0$ for every $a\geq2$ and $H=0$. This contradicts $H=f'\neq0$. Finally, \eqref{eq_ode_k} and \eqref{eq_key_identity} give $P(\xi_i)=0$ for every fiber of \eqref{eq_local_model}.
\end{proof}

\begin{proof}[Proof of Corollary \ref{cor_local_bound}]
	By Theorem \ref{decompwarp}, $M$ is locally isometric to the multiply warped product \eqref{eq_local_model}, whose fibers of dimension at least two are Einstein and whose warping functions satisfy \eqref{eqXiRelFuncWarp}. By \eqref{eq_ode_k} and Lemma \ref{lem_ply_deg2_}, each $\xi_i$ is a root of the polynomial $P$ of \eqref{poly_atmst2}. Suppose $\ell\geq3$. Since $\xi_1,\ldots,\xi_\ell$ are pairwise distinct, $P$ would have at least three distinct roots and hence would vanish identically, contradicting Lemma \ref{lem_ply_deg2_}. Thus there are at most two fibers and, by \eqref{eqXiRelFuncWarp}, the eigenvalues $\lambda_a=\lambda-f'\xi_a$ take at most two distinct values. Therefore $\mathrm{Ric}$ has at most three distinct eigenvalues. The density argument above proves the result.
\end{proof}

\section{Equivalent formulations}

\begin{proof}[Proof of Theorem \ref{thm_equivalence}]
The equivalence of \ref{eq_i}, \ref{eq_ii} and \ref{eq_iii} is Theorem \ref{decompwarp}. We prove the remaining implications in Figure \ref{fig_equivalence}. Following the arrows, every condition leads to \ref{eq_i}, and \ref{eq_i} leads to every condition through \ref{eq_ii}, so all the conditions are equivalent.

\begin{figure}[ht]
\centering
\begin{tikzpicture}[
	cond/.style={inner sep=2pt},
	imp/.style={-{Implies},double,double distance=1.4pt,shorten <=1pt,shorten >=1pt},
	eqv/.style={{Implies}-{Implies},double,double distance=1.4pt,shorten <=1pt,shorten >=1pt}]
	\node[cond] (i)    at (90:2.8)   {\ref{eq_i}};
	\node[cond] (ii)   at (18:2.8)   {\ref{eq_ii}};
	\node[cond] (iv)   at (-54:2.8)  {\ref{eq_iv}};
	\node[cond] (v)    at (-126:2.8) {\ref{eq_v}};
	\node[cond] (vii)  at (162:2.8)  {\ref{eq_vii}};
	\node[cond] (iii)  at (54:4.3)   {\ref{eq_iii}};
	\node[cond] (viii) at (-18:4.3)  {\ref{eq_viii}};
	\node[cond] (vi)   at (18:4.7)   {\ref{eq_vi}};
	\draw[imp] (i) to[bend left=36] (ii);
	\draw[imp] (ii) to[bend left=36] (iv);
	\draw[imp] (iv) to[bend left=36] (v);
	\draw[eqv] (v) to[bend left=36] (vii);
	\draw[imp] (vii) to[bend left=36] (i);
	\draw[imp] (ii) to[bend right=15] (iii);
	\draw[imp] (iii) to[bend right=15] (i);
	\draw[imp] (ii) to[bend left=15] (viii);
	\draw[imp] (viii) to[bend left=15] (iv);
	\draw[eqv] (ii) -- (vi);
\end{tikzpicture}
\caption{The implications in the proof of Theorem \ref{thm_equivalence}.}\label{fig_equivalence}
\end{figure}

\emph{\ref{eq_ii}$\Rightarrow$\ref{eq_iv}.} This is \eqref{eq_ode_k}.

\emph{\ref{eq_iv}$\Rightarrow$\ref{eq_v}.} By Lemma \ref{lem_ply_deg2_}, condition \ref{eq_iv} says that every $\xi_a$ is a root of $P$. If the $\xi_a$ took three or more distinct values, $P$ would vanish identically, contradicting the same lemma. Thus the $\xi_a$ take at most two distinct values. Since they are the principal curvatures of the level sets and are constant on them by Lemma \ref{lem_lambda_level}, condition \ref{eq_v} follows.

\emph{\ref{eq_v}$\Leftrightarrow$\ref{eq_vii}.} The principal curvatures of the level sets are $\xi_a=(\lambda-\lambda_a)/f'$. Since $f'\neq0$, the eigenvalues of $\mathrm{Ric}$ tangent to the level sets take at most two distinct values if and only if the level sets have at most two distinct principal curvatures. As observed above, the level sets are always isoparametric.

\emph{\ref{eq_vii}$\Rightarrow$\ref{eq_i}.} Condition \ref{eq_vii} says that $\ell\leq2$. We verify local decomposability directly using the criterion stated after Definition \ref{def_loc_diff_prod}. The relations \eqref{eq_preserve_Da} give $[E_1,D_a]\subset D_a$, while the integrability of each $D_a$ follows from Lemma \ref{derdlemma}\eqref{derd4}. If $[a]\neq[\alpha]$, $X\in D_a$ and $Y\in D_\alpha$, then Lemma \ref{nablaframe} gives
\[
	\left\langle[X,Y],E_1\right\rangle=\left\langle\nabla_XY,E_1\right\rangle-\left\langle\nabla_YX,E_1\right\rangle=-\xi_a\langle Y,X\rangle+\xi_\alpha\langle X,Y\rangle=0.
\]
Thus the only possible obstruction would be a component of $[X,Y]$ in a third tangent eigendistribution. Since $\ell\leq2$, no such distribution exists, and $\mathcal{N}$ is locally decomposable.

\emph{\ref{eq_ii}$\Rightarrow$\ref{eq_vi}.} In the local model \eqref{eq_local_model}, a level set is the Riemannian product $(N_1,h_1^2g_{N_1})\times\cdots\times(N_\ell,h_\ell^2g_{N_\ell})$, the $h_i$ being constant on it. Its Ricci tensor is $\mathrm{Ric}_{\Sigma_s}=\sum_ik_i\,g_i$ with $k_i=(r_i-1)\mu_i/h_i^{2}$ constant by \eqref{first_scnd_Dg}, and $g_i$ the restriction of the metric to the $i$ th factor. On a Riemannian product each $g_i$ is parallel, so $\nabla^{\Sigma_s}\mathrm{Ric}_{\Sigma_s}=0$.

\emph{\ref{eq_ii}$\Rightarrow$\ref{eq_viii}.} In the same local model, distinct tangent Ricci eigendistributions are tangent to distinct factors of the Riemannian product level set. Hence all the corresponding mixed sectional curvatures vanish.

\emph{\ref{eq_viii}$\Rightarrow$\ref{eq_iv}.} Fix a point and choose the orthonormal Ricci eigenframe so that $\nabla_{E_1}E_i=0$ there, which is possible by \eqref{eq_preserve_Da}. All derivatives below are evaluated at this point. For $a,j\geq2$ with $a\neq j$, set
\[
	K_{aj}=\mathrm{Rm}(E_a,E_j,E_a,E_j),\qquad
	\widehat K_{aj}=\mathrm{Rm}_{\Sigma_s}(E_a,E_j,E_a,E_j).
\]
Since $R$ is constant on each level set, \eqref{eq_bari} gives $\mathrm{Rm}(E_1,X,Y,Z)=0$ for $X,Y,Z$ tangent to $\Sigma_s$. Differentiating this identity and using Lemma \ref{nablaframe} and \eqref{eq_radial_curvature}, we obtain from the second Bianchi identity
\begin{align*}
	K_{aj}'
	&=(\nabla_{E_a}\mathrm{Rm})(E_1,E_j,E_a,E_j)
	  +(\nabla_{E_j}\mathrm{Rm})(E_a,E_1,E_a,E_j)\\
	&=-\xi_a(K_{aj}+q)-\xi_j(K_{aj}+q).
\end{align*}
By the Gauss equation, $\widehat K_{aj}=K_{aj}+\xi_a\xi_j$. Therefore, using \eqref{excep},
\begin{align*}
	\widehat K_{aj}'
	&=-(\xi_a+\xi_j)(K_{aj}+q)
	  +(q-\xi_a^2)\xi_j+\xi_a(q-\xi_j^2)\\
	&=-(\xi_a+\xi_j)\widehat K_{aj}.
\end{align*}
Taking the trace gives
\[
	k_a'+2\xi_ak_a
	=\sum_{\substack{j\geq2\\j\neq a}}(\xi_a-\xi_j)\widehat K_{aj}=0.
\]
Indeed, if $j\in[a]$, then $\xi_j=\xi_a$; otherwise, $\widehat K_{aj}=0$ by \ref{eq_viii}. This proves \ref{eq_iv}.

\emph{\ref{eq_vi}$\Rightarrow$\ref{eq_ii}.} If $\ell\leq2$, the argument in the proof of \ref{eq_vii}$\Rightarrow$\ref{eq_i} shows that the Ricci net is locally decomposable, and Theorem \ref{decompwarp} gives the conclusion. Suppose then that $\ell\geq3$. It is enough to work near a level set $\Sigma_{s_0}$. Choose an orthonormal Ricci eigenframe $\{E_i\}$ which is parallel along the $E_1$ flow inside the tangent eigendistributions:
\[
	\nabla_{E_1}E_i=0,\qquad i\geq2.
\]
This is possible by \eqref{eq_preserve_Da}. In this frame,
\[
	[E_1,E_i]=-\xi_iE_i,\qquad i\geq2.
\]

Let $a$, $A$ and $\alpha$ represent three pairwise distinct tangent Ricci eigendistributions and set
\[
	C_{aA}^{\alpha}=\left\langle\nabla_{E_a}E_A,E_\alpha\right\rangle .
\]
Lemma \ref{derdlemma}\eqref{derd2} gives
\[
	\left\langle[E_a,E_A],E_\alpha\right\rangle
	=\frac{\xi_a-\xi_A}{\xi_a-\xi_\alpha}C_{aA}^{\alpha}.
\]
Using Jacobi's identity, the fact that $E_i(\xi_j)=0$ for $i,j\geq2$, and
\[
	(\xi_i-\xi_j)'=-(\xi_i+\xi_j)(\xi_i-\xi_j),
\]
which follows from \eqref{excep}, we obtain
\begin{align}\label{eq_C_transport}
	(C_{aA}^{\alpha})'=-\xi_aC_{aA}^{\alpha}.
\end{align}

On $\Sigma_{s_0}$, \eqref{eq_full_ricci_decomposition} shows that the same frame diagonalizes $\mathrm{Ric}_{\Sigma_{s_0}}$. Since this tensor is parallel, the local de Rham--Wu decomposition theorem gives, after shrinking the level set if necessary, a Riemannian product of Einstein factors
\[
	\Sigma_{s_0}=L_1\times\cdots\times L_m
\]
whose factors are tangent to the eigendistributions of $\mathrm{Ric}_{\Sigma_{s_0}}$. In particular,
\[
	(k_j-k_\ell)\left\langle\nabla^{\Sigma_{s_0}}_{E_i}E_j,E_\ell\right\rangle=0
\]
for tangent indices $i,j,\ell$.

The numbers $k_a$, $k_A$ and $k_\alpha$ are the values of the quadratic polynomial $x\mapsto C+Bx-x^2$ at the three distinct numbers $\xi_a$, $\xi_A$ and $\xi_\alpha$, and hence they cannot all be equal. If $k_A\neq k_\alpha$, the preceding identity gives $C_{aA}^{\alpha}=0$ on $\Sigma_{s_0}$. If $k_A=k_\alpha$, then $k_a\neq k_A$, and the same identity gives $C_{Aa}^{\alpha}=0$. Applying Lemma \ref{derdlemma}\eqref{derd2} once more gives $C_{aA}^{\alpha}=0$. Therefore these coefficients vanish on $\Sigma_{s_0}$ for every triple of distinct tangent Ricci classes.

It follows from \eqref{eq_C_transport} that they vanish in a neighborhood of $\Sigma_{s_0}$. Consequently,
\[
	[D_a,D_A]\subset D_a\oplus D_A,\qquad [E_1,D_a]\subset D_a.
\]
Together with the integrability of each $D_a$, the local decomposability criterion stated after Definition \ref{def_loc_diff_prod} shows that the Ricci net is locally decomposable. Theorem \ref{decompwarp} now gives the multiply warped product structure.
\end{proof}

\begin{proof}[Proof of Corollary \ref{cor_unconditional}]
Kim proved, more precisely, that the eigenvalues of the Ricci tensor tangent to the level sets take at most two distinct values (\cite[Proposition 3.4]{kim1} for $n=4$ and \cite{kim3} for $n\geq5$). This is condition \ref{eq_vii} of Theorem \ref{thm_equivalence}, and the other conditions follow.
\end{proof}

\begin{remark}\label{rem_extensions}
It is worth mentioning that some of the arguments above are not specific to the constancy of $\lambda$ in \eqref{fundeq}. With the corresponding changes in notation, the equivalences $\ref{eq_i}\Leftrightarrow\ref{eq_ii}\Leftrightarrow\ref{eq_iii}$, the implications $\ref{eq_ii}\Rightarrow\ref{eq_iv}$ and $\ref{eq_ii}\Leftrightarrow\ref{eq_vi}$, and the equivalence $\ref{eq_v}\Leftrightarrow\ref{eq_vii}$ also hold for gradient almost Ricci solitons and for quasi-Einstein manifolds with harmonic Weyl curvature. 
\end{remark}

\section{Acknowledgments}
V. Borges was partially supported by Capes Finance Code 001, and he thanks the Mathematics Department of Universidade de Brasília, where this work was conducted. M. A. R. M. Horácio was supported by Capes Finance Code 001. J. P. dos Santos was partially supported by CNPq 313200/2025-4 and FAPDF 00193-00001678/2024-39.

\end{document}